\documentclass[10pt,oneside]{amsart}

\usepackage{amssymb, amsmath, amsthm, amsfonts}
\usepackage{mathrsfs,comment, mathtools}
\usepackage{setspace}
\usepackage{stmaryrd}
\input{xypic}
\usepackage{graphicx}
\usepackage{geometry}
\usepackage{float}
\usepackage[all]{xy}

\usepackage[bookmarks,bookmarksopen,bookmarksdepth=4]{hyperref}
\usepackage{url}
\usepackage{enumerate}
\usepackage{color}

\usepackage{verbatim}
\usepackage[T1]{fontenc}
\usepackage{todonotes}

\hypersetup{%
  bookmarksnumbered=true,%
  colorlinks=true,%
  linkcolor=blue,%
  citecolor=blue,%
  filecolor=blue,%
  menucolor=blue,%
  urlcolor=blue,%
  pdfnewwindow=true,%
  pdfstartview=FitBH}

\def\@url#1{{\tt\def~{\lower3.5pt\hbox{\char'176}}\def\_{\char'137}#1}}

\newtheorem{theorem}{Theorem}[section]
\newtheorem{corollary}[theorem]{Corollary}
\newtheorem{proposition}[theorem]{Proposition}
\newtheorem{lemma}[theorem]{Lemma}

\theoremstyle{definition}
\newtheorem{definition}[theorem]{Definition}
\newtheorem{example}[theorem]{Example}

\newtheorem{remark}[theorem]{Remark}

\newtheorem{conjecture}[theorem]{Conjecture}

\theoremstyle{plain}

\newtheorem{xxthm}{Theorem}

\newtheorem{xxcor}[xxthm]{Corollary}
\newtheorem{xxconj}[xxthm]{Conjecture}

\newcommand{\Mdef}[2]{\newcommand{#1}{\relax \ifmmode #2 \else $#2$\fi}}
\Mdef{\bA}{\mathbb{A}}
\Mdef{\bB}{\mathbb{B}}
\Mdef{\bC}{\mathbb{C}}
\Mdef{\bD}{\mathbb{D}}
\Mdef{\bE}{\mathbb{E}}
\Mdef{\bF}{\mathbb{F}}
\Mdef{\bG}{\mathbb{G}}
\Mdef{\bH}{\mathbb{H}}
\Mdef{\bI}{\mathbb{I}}
\Mdef{\bJ}{\mathbb{J}}
\Mdef{\bK}{\mathbb{K}}
\Mdef{\bL}{\mathbb{L}}
\Mdef{\bM}{\mathbb{M}}
\newcommand{\bN}{\mathbb{N}}
\Mdef{\bO}{\mathbb{O}}
\Mdef{\bP}{\mathbb{P}}
\Mdef{\bQ}{\mathbb{Q}}
\Mdef{\bR}{\mathbb{R}}
\Mdef{\bS}{\mathbb{S}}
\Mdef{\bT}{\mathbb{T}}
\Mdef{\bU}{\mathbb{U}}
\Mdef{\bV}{\mathbb{V}}
\Mdef{\bW}{\mathbb{W}}
\Mdef{\bX}{\mathbb{X}}
\Mdef{\bY}{\mathbb{Y}}
\Mdef{\bZ}{\mathbb{Z}}

\Mdef{\bbS}{\mathbb{S}}

\Mdef{\scrA}{\mathscr{A}}
\Mdef{\scrB}{\mathscr{B}}
\Mdef{\scrC}{\mathscr{C}}
\Mdef{\scrD}{\mathscr{D}}
\Mdef{\scrE}{\mathscr{E}}
\Mdef{\scrF}{\mathscr{F}}
\Mdef{\scrG}{\mathscr{G}}
\Mdef{\scrH}{\mathscr{H}}
\Mdef{\scrI}{\mathscr{I}}
\Mdef{\scrJ}{\mathscr{J}}
\Mdef{\scrK}{\mathscr{K}}
\Mdef{\scrL}{\mathscr{L}}
\Mdef{\scrM}{\mathscr{M}}
\Mdef{\scrN}{\mathscr{N}}
\Mdef{\scrO}{\mathscr{O}}
\Mdef{\scrP}{\mathscr{P}}
\Mdef{\scrQ}{\mathscr{Q}}
\Mdef{\scrR}{\mathscr{R}}
\Mdef{\scrS}{\mathscr{S}}
\Mdef{\scrT}{\mathscr{T}}
\Mdef{\scrU}{\mathscr{U}}
\Mdef{\scrV}{\mathscr{V}}
\Mdef{\scrW}{\mathscr{W}}
\Mdef{\scrX}{\mathscr{X}}
\Mdef{\scrY}{\mathscr{Y}}
\Mdef{\scrZ}{\mathscr{Z}}

\Mdef{\mcA}{\mathcal{A}}
\Mdef{\mcB}{\mathcal{B}}
\Mdef{\mcC}{\mathcal{C}}
\Mdef{\mcD}{\mathcal{D}}
\Mdef{\mcE}{\mathcal{E}}
\Mdef{\mcF}{\mathcal{F}}
\Mdef{\mcG}{\mathcal{G}}
\Mdef{\mcH}{\mathcal{H}}
\Mdef{\mcI}{\mathcal{I}}
\Mdef{\mcJ}{\mathcal{J}}
\Mdef{\mcK}{\mathcal{K}}
\Mdef{\mcL}{\mathcal{L}}
\Mdef{\mcM}{\mathcal{M}}
\Mdef{\mcN}{\mathcal{N}}
\Mdef{\mcO}{\mathcal{O}}
\Mdef{\mcP}{\mathcal{P}}
\Mdef{\mcQ}{\mathcal{Q}}
\Mdef{\mcR}{\mathcal{R}}
\Mdef{\mcS}{\mathcal{S}}
\Mdef{\mcT}{\mathcal{T}}
\Mdef{\mcU}{\mathcal{U}}
\Mdef{\mcV}{\mathcal{V}}
\Mdef{\mcW}{\mathcal{W}}
\Mdef{\mcX}{\mathcal{X}}
\Mdef{\mcY}{\mathcal{Y}}
\Mdef{\mcZ}{\mathcal{Z}}

\Mdef{\tA}{\tilde{A}}
\Mdef{\tB}{\tilde{B}}
\Mdef{\tC}{\tilde{C}}
\Mdef{\tE}{\tilde{E}}
\Mdef{\tH}{\tilde{H}}
\Mdef{\tK}{\tilde{K}}
\Mdef{\tL}{\tilde{L}}
\Mdef{\tM}{\tilde{M}}
\Mdef{\tN}{\tilde{N}}
\Mdef{\tP}{\tilde{P}}

\Mdef{\Ab}{\overline{A}}
\Mdef{\Bb}{\overline{B}}
\Mdef{\Cb}{\overline{C}}
\Mdef{\Db}{\overline{D}}
\Mdef{\Eb}{\overline{E}}
\Mdef{\Fb}{\overline{F}}
\Mdef{\Gb}{\overline{G}}
\Mdef{\Hb}{\overline{H}}
\Mdef{\Ib}{\overline{I}}
\Mdef{\Jb}{\overline{J}}
\Mdef{\Kb}{\overline{K}}
\Mdef{\Lb}{\overline{L}}
\Mdef{\Mb}{\overline{M}}
\Mdef{\Nb}{\overline{N}}
\Mdef{\Ob}{\overline{O}}
\Mdef{\Pb}{\overline{P}}
\Mdef{\Qb}{\overline{Q}}
\Mdef{\Rb}{\overline{R}}
\Mdef{\Sb}{\overline{S}}
\Mdef{\Tb}{\overline{T}}
\Mdef{\Ub}{\overline{U}}
\Mdef{\Vb}{\overline{V}}
\Mdef{\Wb}{\overline{W}}
\Mdef{\Xb}{\overline{X}}
\Mdef{\Yb}{\overline{Y}}
\Mdef{\Zb}{\overline{Z}}

\newcommand{\co}{\colon}

\def\endash{\mathchar"2D}

\newcommand{\leftmod}{\endash \textnormal{mod}}

\newcommand{\lra}{\longrightarrow}

\newcommand{\spO}{\mathrm{Sp}^{O}}
\newcommand{\GspO}{G\mathrm{Sp}^{O}}
\newcommand{\QGspO}{G\mathrm{Sp}^{O}_\bQ}
\newcommand{\eqtop}[1]{#1 \mathrm{Top_*}}

\newcommand{\Ch}{\mathrm{Ch}}

\Mdef{\id}{\mathrm{Id}}

\newcommand{\adjunction}[4]{
\xymatrix{
#1:#2 \ar@<0.7ex>[r] &
\ar@<0.7ex>[l] #3:#4
}}

\newcommand{\rank}{\mathrm{rank}}

\newcommand{\smashprod}{\wedge}

\Mdef{\bhom}{\mathbf{\hat{H}om}}

\Mdef{\Mod}{\mathrm{mod}}

\newcommand{\Ho}{\mathrm{Ho}}

\newcommand{\colim}{\mathop{  \mathrm{colim }} }

\newcommand{\weylsheaf}[1]{\textnormal{Weyl} \endash #1 \endash \textnormal{sheaf}_{\bQ}(\sub #1)}

\newcommand{\Rweylsheaf}[2]{\textnormal{Weyl} \endash #1 \endash \textnormal{sheaf}_{#2}(\sub #1)}

\newcommand{\adic}[1]{\bZ_{#1}^{\wedge}}

\newcommand{\mackey}[1]{\textnormal{Mackey}_{\bQ}(#1)}
\newcommand{\intmackey}[1]{\textnormal{Mackey}(#1)}

\newcommand{\sub}{\mathcal{S}}

\newcommand{\Gsetsdiscfin}{G_{df} \endash \textnormal{sets}}

\newcommand{\abgps}{\textnormal{Ab}}

\newcommand{\opensub}{\underset{\textnormal{open}}{\leqslant}}

\newcommand{\opennormalsub}{\underset{\textnormal{open}}{\trianglelefteqslant}}

\DeclareMathOperator{\stab}{\textnormal{stab}}

\DeclareMathOperator{\coker}{\textnormal{Coker}}

\DeclareMathOperator{\const}{\textnormal{Const}}

\DeclareMathOperator{\core}{\textnormal{Core}}

\newcommand{\intburnsidering}[1]{\textbf{\textsf{\upshape{A}}}(#1)} 
\newcommand{\ratburnsidering}[1]{\textbf{\textsf{\upshape{A}}}_{\bQ}(#1)}

\newcommand{\piorbcat}{\pi_0(\mathfrak{O}_G)}

\newcommand{\piratorbcat}{\pi_0(\mathfrak{O}^\bQ_G)}

\newcommand{\spanunstableorbcat}{\textnormal{Span}(\Gsetsdiscfin)}

\newcommand{\spanunstableorbcatN}[1]{\textnormal{Span}(#1_{df} \endash \textnormal{sets})}

\newcommand{\func}{\textnormal{Func}}
\DeclareMathOperator{\ext}{Ext}

\title{Classifying rational \texorpdfstring{$G$}{G}-spectra for profinite \texorpdfstring{$G$}{G}}
\author[Barnes]{David Barnes}
\address[Barnes]{Mathematical Sciences Research Centre, Queen's University Belfast}
\author[Sugrue]{Danny Sugrue}
\address[Sugrue]{Mathematical Sciences Research Centre, Queen's University Belfast}

\thanks{
The second author gratefully acknowledges support from the
Engineering and Physical Sciences Research Council under Grant 1631308.
No data were created or analysed in this study.
}

\begin{document}
\begin{abstract}
For $G$ an arbitrary profinite group, we construct an algebraic model for rational $G$-spectra in terms of 
$G$-equivariant sheaves over the space of subgroups of $G$. 
This generalises the known case of finite groups to a much wider class of topological groups,
and improves upon earlier work of the first author on the case where $G$ is the $p$-adic integers.  

As the purpose of an algebraic model is to allow one to use homological algebra to study 
questions of homotopy theory, we prove that the homological dimension (injective dimension)
of the algebraic model is determined by the Cantor--Bendixson rank of the 
space of closed subgroups of the profinite group $G$. 
This also provides a calculation of the homological dimension of the category of rational Mackey functors.
\end{abstract}

\maketitle

\tableofcontents

\section{Introduction}

The usefulness of equivariant cohomology theories in equivariant (stable) homotopy theory has long been proven. 
Examples of equivariant cohomology theories include the equivariant $K$-theory of Segal, \cite{segal68}
and the equivariant cobordism spectra used in Hill, Hopkins and Ravenel, \cite{HillHopkinsRavenel}.
To effectively study equivariant cohomology theories one studies the 
category of their representing objects: equivariant spectra. That is,
Brown representability holds equivariantly.  

The study of equivariant spectra up to homotopy is even more demanding than the non-equivariant case, 
so one often works with rational equivariant spectra. 
Under Brown representability, rational equivariant spectra correspond to equivariant cohomology
theories that take values in rational vector spaces.
Rationalising preserves most of the interesting behaviour coming from the group, 
while removing much of the topological complexity. 

A major goal in the study of rational equivariant stable homotopy theory is to find 
a more tractable model category that has the same homotopy theory as rational equivariant spectra.
That is, one chooses a group $G$ of interest, constructs an 
abelian category $\mcA(G)$ and a Quillen equivalence 
between the \emph{algebraic model} $\Ch( \mcA(G) )$ 
and the model category of rational $G$--spectra.  
The Quillen equivalence induces an equivalence of categories between the 
homotopy category of rational $G$--spectra and the homotopy category of the 
algebraic model. 
The primary advantage of having an algebraic model is that 
one can use the simplicity of the abelian category $\mcA(G)$
and the tools of homological algebra
to construct objects and calculate sets of maps in 
the rational $G$-equivariant stable homotopy category. 
For an introduction to algebraic models and summary of the known cases see \cite{BKalgmodels} 
by the first author and K\k{e}dziorek.

In this paper, the authors generalise the known case of algebraic models for finite groups
(see Schwede and Shipley in \cite[Example 5.1.2]{ss03stabmodcat}) to profinite groups. 
Profinite groups are a commonly encountered class of compact topological groups, 
appearing most often as Galois groups or when one has a diagram of finite groups. 
They are defined as the compact Hausdorff totally disconnected topological groups
and it can be shown that a group is profinite if and only if 
it is the limit of a filtered system of finite groups. 
As a first example, the Morava stabilizer group $\mathbb{S}_n$ 
from chromatic homotopy theory is profinite.
Secondly, number theory makes substantial use of profinite groups, as seen in 
Bley and Boltje \cite{BB04}.
Thirdly, the \'{e}tale fundamental groups of algebraic geometry are profinite. 
This ubiquity drives our interest in rational $G$-spectra for profinite $G$ and hence
our interest in finding algebraic models in the profinite case.

\subsection{Main results}

Let $G$ be a profinite group and let $\sub G$ be the space of closed subgroups of $G$, 
topologised as the inverse limit of the finite discrete spaces $\sub G/N$, for $N$ 
open and normal in $G$. There is an abelian category of rational $G$-equivariant sheaves
over $\sub G$. Consider the full subcategory of those equivariant sheaves $E$ such 
that the stalk $E_K$ over the closed subgroup $K$ is $K$-fixed. 
These are called rational Weyl-$G$-sheaves and are introduced in earlier work of the authors, 
\cite{BSmackey} and \cite{BSsheaves}.
In this paper, the category of rational Weyl-$G$-sheaves occur as the abelian category $\mcA(G)$ 
used to make the algebraic model $\Ch( \mcA(G) )$ for rational $G$-spectra.

\begin{xxthm}[Corollary \ref{cor:gpsectrasheaf}]
For $G$ a profinite group, there is a zig-zag of 
Quillen equivalences between the model category of rational $G$-spectra
and the model category of chain complexes of rational Weyl-$G$-sheaves. 
\end{xxthm}

This result generalises work of the first author \cite{barnespadic} 
on the case where $G=\adic{p}$ is the $p$-adic integers. As with the current work,
that paper uses the tilting theory (Morita theory) of 
Schwede and Shipley \cite[Theorem 5.1]{ss03stabmodcat} to obtain a
Quillen equivalence between rational $G$-spectra and a category of 
chain complexes in an abelian category. 
The $p$-adic case uses a hand-crafted abelian 
category designed for that specific group. Contrastingly, the current work 
uses rational $G$-Mackey functors in the tilting theory step, then applies 
the equivalence between rational $G$-Mackey functors and rational 
Weyl-$G$-sheaves of \cite[Theorem A]{BSmackey} to obtain the final result.

There are two reasons why the sheaf description of the algebraic model is 
important. The first is that Greenlees' conjecture on 
algebraic models (for compact Lie groups) \cite{Gconjecture}
is described in terms of sheaves over $\sub G$, the space of closed subgroups of $G$.
This result is the first realisation of that conjecture, albeit for 
profinite groups. 

The second is that using the sheaf description we can calculate the homological dimension
(also called the injective dimension) of the abelian category. 
This dimension is a measure of the complexity of the abelian model.
In this case the result is phrased in terms of the 
Cantor--Bendixson rank of the space of subgroups $\sub G$, see 
Section \ref{sec:injdim} for details. 
When $G$ is profinite, $\sub G$ is a profinite space (one that is
compact, Hausdorff and totally disconnected).
The Cantor--Bendixson rank of a profinite space can be thought of as a measure of how far the 
space is from being discrete. To illustrate, 
a discrete space has rank 1 and $\sub \adic{p}$, 
consisting of countably many points with one accumulation point, has rank 2.

\begin{xxthm}[Corollary \ref{cor:injectivedimmain}]
Let $G$ be a profinite group whose space of subgroups $\sub G$ is scattered of Cantor--Bendixson rank $n$.
The homological dimension of the category of rational Weyl-$G$-sheaves is $n-1$. 

If $\sub G$ has infinite Cantor--Bendixson rank, then the homological dimension 
of the category of rational Weyl-$G$-sheaves is infinite.
\end{xxthm}

Using the equivalence between rational $G$-Mackey functors 
and rational Weyl-$G$-sheaves, \cite[Theorem A]{BSmackey}, 
this result gives the homological dimension of categories
of rational $G$-Mackey functors. 

\begin{xxcor}
Let $G$ be a profinite group whose space of subgroups $\sub G$ is scattered of Cantor--Bendixson rank $n$.
The homological dimension of the category of rational $G$-Mackey functors is $n-1$. 

If $\sub G$ has infinite Cantor--Bendixson rank, then the homological dimension 
of the category of rational $G$-Mackey functors is infinite.
\end{xxcor}

As well as the two cases in the previous results there is a third possibility,
that $\sub G$ is of finite rank, but not scattered. We make the 
following conjecture for this case, which occurs as Conjecture \ref{conj:hull} in the main body.

\begin{xxconj}
Let $G$ be a profinite group. If the $G$-space $X$ has finite Cantor--Bendixson rank and non-empty perfect hull, then the homological dimension 
of rational $G$-sheaves over $X$ is infinite.

If $X=\sub G$, then the homological dimension 
of the category of rational Weyl-$G$-sheaves is infinite.
\end{xxconj}

\subsection{Future questions}
The question of how the change of groups functors on spectra compare with 
functors relating algebraic models for varying $G$ is surprisingly involved. 
The currently known cases are for (co)free equivariant spectra, see 
Williamson \cite{Williamson22}. 
For the case of (pro)finite groups, the authors expect that having the 
sheaf description and the Mackey functor description will be vital.

The Quillen equivalences as given are not monoidal. 
There are two sources of difficulty, firstly that the Quillen equivalences of the 
tilting theory of Schwede and Shipley are not monoidal. 
Resolving this would require a fundamentally different approach to 
the classification. 
Secondly, it is not known how the equivalence of \cite{BSmackey} 
interacts with the tensor product of sheaves 
and the two known monoidal structures on $G$-Mackey functors: 
the box product and the equivariant tensor product of 
Hill and Mazur, \cite{HMequivarianttensor}.

A further question is whether one can construct an Adams spectral sequence
which takes values in the abelian category. The difficulty here is 
expected to be around constructing a suitable set of 
geometric fixed point functors for the closed subgroups of $G$,
which should be the topological equivalent of taking the stalk 
of an equivariant sheaf and hence detect equivalences in the homotopy category of rational $G$-spectra.  
The geometric fixed popint functors and the Adams spectral sequence are needed to 
give a good set of examples of rational $G$-spectra and their image in the 
algebraic model.

\subsection{Strategy of the classification}
The following diagram gives the major steps in the classification of 
rational $G$-spectra, for profinite $G$.
The Tilting Theorem (Morita theory) of 
Schwede and Shipley \cite[Theorem 5.1]{ss03stabmodcat} is used in 
Subsection \ref{subsec:tilting} to create a 
Quillen equivalence between rational $G$-spectra and a category of (chain complexes of) 
``topological Mackey functors''. This gives the upper horizontal functor.
The key input to apply the Tilting Theorem is Theorem \ref{thm:homotopydegreezero}, 
which proves that the homotopical information of rational $G$-spectra
to be concentrated in degree zero.

In Section \ref{sec:spans} we study spans and the stable orbit category.
The aim is to prove Theorem \ref{thm:spantoorbit}, which gives an equivalence 
between $\piorbcat$, the $G$-equivariant stable orbit category
and $\spanunstableorbcat$, a category of spans of finite discrete $G$-sets. 
That equivalence provides the upper vertical functor of the diagram. 
The lower vertical functor is an equivalence of categories 
describing Mackey functors in terms of spans, as detailed in 
Subsection \ref{subsec:mackey}. 
Theorem \ref{thm:gpsectramackey} is where these results are combined to 
give the zig-zag of Quillen equivalences between rational $G$-spectra and 
chain complexes of rational $G$-Mackey functors.

The lower horizontal functor is an equivalence by earlier work of the authors, 
\cite[Theorem A]{BSmackey}, which proves that 
the category of rational $G$-Mackey functors 
is equivalent to the category of rational Weyl-$G$-sheaves over the 
space of closed subgroups of $G$. 

\begin{align*}\label{maindiagram}
\xymatrix{
\text{Rational $G$-spectra}
\ar@{<->}[r]^-{\simeq} & 
\Ch \left(\func_{\abgps}\left(\piratorbcat,\bQ \leftmod \right)\right)
\ar@{<->}[d]^-{\cong} 
\\
& \Ch\left(\func_{\abgps}\left(\spanunstableorbcat ,\bQ \leftmod\right) \right) 
\ar@{<->}[d]^-{\cong}
\\
\Ch \left(\weylsheaf{G}\right)
& 
\Ch \left(\mackey{G}\right)
\ar@{<->}[l]_-{\cong} 
}
\end{align*}

In Section \ref{sec:injdim} we look at the homological dimension of the algebraic model
and relate it to the Cantor-Bendixson rank of $\sub G$.

\section{Basics on equivariant spectra for profinite \texorpdfstring{$G$}{G}}
We recap the construction of the model category of rational orthogonal $G$-spectra, for $G$ a profinite group. 
We then give various properties of this homotopy theory that will be used in the classification. 
Along the way we will need some facts about profinite groups and sets and topological spaces with a profinite group action.
The constructions will all be generalisations of the finite group case. 
The expert reader may like to skip to Section \ref{sec:spans}.

\subsection{Profinite groups}

We give a few reminders of useful facts on profinite groups.
More details can be found in Wilson \cite{wilson98} or Ribes and Zalesskii \cite{rz00}.

A \emph{profinite group} is a compact, Hausdorff, totally disconnected
topological group.
A profinite group $G$ is homeomorphic to the inverse limit of its finite quotients:
\[
G \cong \underset{N \opennormalsub G}{\lim}  G/N \subseteq \underset{N \opennormalsub G}{\prod} G/N .
\]
The limit has the canonical topology which either be described as the subspace topology on the product
or as the topology generated by the preimages of the open sets in
$G/N$ under the projection map $G \to G/N$, as $N$ runs over the open normal subgroups of $G$.

Closed subgroups and quotients by closed subgroups of profinite groups are also profinite.
A subgroup of a profinite group is open if and only if it is finite index and closed.
The trivial subgroup $\{ e\}$ is open if and only
if the group is finite.
The intersection of all open normal subgroups is $\{ e \}$.
Any open subgroup $H$ contains an open normal subgroup, the \emph{core} of $H$ in $G$,
which is defined as the finite intersection
\[
\core_G (H) = \bigcap_{g \in G} g H g^{-1}.
\]

\subsection{Equivariant orthogonal spectra}
Recall that an orthogonal spectrum is a sequence of based spaces indexed by 
finite dimensional inner product spaces, related by suspension maps that are suitably compatible
with linear isometries of those vector spaces. For details, see Mandell, May Schwede and Shipley \cite{mmss01} 
or work of the first author with Roitzheim \cite{brfoundations}. 

Equivariantly, the picture is similar, starting from based spaces with $G$-action indexed
by finite dimensional $G$-inner product spaces, with suspension maps that are compatible
with both the linear isometries and $G$-actions. 
This construction was first given in work of Mandell and May \cite{mm02} and adapted 
to the profinite setting by Fausk \cite{fausk}.

The starting point is to describe the model category of based topological $G$-spaces
that we will use to create our $G$-spectra. 
We focus on those model structures built using the \emph{open} subgroups $H$ of $G$
as then $G/H$ is a finite set. This ensures that the forgetful functor from
$G$-spaces to spaces is a left Quillen functor. Throughout this section $G$ 
will be a profinite group. 

\begin{proposition}
There is a cofibrantly generated proper model structure on the category
of based topological $G$-spaces with weak equivalences those maps $f$ such that $f^H$ is a
weak equivalence of spaces, for all open subgroups $H$ of $G$. 
Similarly, fibrations are those maps $f$ such that $f^H$ is a fibration of based spaces 
for all open subgroups $H$ of $G$. 
This model structure is denoted $\eqtop{G}$.

The generating cofibrations are the standard inclusions 
\begin{align*}
G/H_+ \smashprod S^{n-1}_+  &\lra G/H_+ \smashprod D^n_+ \\
G/H_+ \smashprod D^n_+ &\lra G/H_+ \smashprod (D^n \times [0,1])_+
\end{align*}
for $H$ an open subgroup of $G$ and $n \geqslant 0$. 
\end{proposition}

Just as CW-complexes are built from iteratively attaching cells 
by taking the pushout over the inclusion $S^{n-1} \to D^n$, one can define
$G$-CW complexes using the inclusions 
\[
G/H \times S^{n-1} \lra G/H \times D^n.
\]
That is, take $X_0$ to be a disjoint union of copies of $G/H_+$, 
then attach cells of the above form with $n=1$ and $H$ varying to obtain $X_1$. 
Continuing inductively gives $X_n$ and $X$ is defined as the union of the $X_n$. 
We see that if $X$ is built using finitely many cells, the stabiliser of $X$ 
(the intersection of all open subgroups $H$ used in the cells) is also open. 
The evident pointed analogue gives the definition of pointed $G$-CW-complexes.
Since the spaces $G/H$ are finite, we see that every $G$-CW complex 
is indeed a CW-complex, after forgetting the group action. 
Note that our choices mean that $G$ itself is not a $G$-CW complex.
Indeed, as it has no fixed points, the space $G$ is weakly equivalent to the empty set.

Our category of $G$-spectra will be indexed by the finite dimensional sub-$G$-inner product spaces 
of a complete $G$-universe $\mcU$, as defined below.

\begin{definition}
A \emph{$G$-universe} $\mcU$ is a countably infinite direct sum $\mcU= \oplus_{i=1}^\infty U$
of real a $G$-inner product space $U$, such that:
\begin{enumerate}
\item there is a canonical choice of trivial representation $\bR \subset U$, 
\item $\mcU$ is topologised as the union of all finite dimensional 
$G$-subspaces of $U$ (each equipped with the norm topology).
\end{enumerate}
A $G$-universe is said to be \emph{complete} if every finite dimensional irreducible representation
is contained (up to isomorphism) within $\mcU$. 
\end{definition}

A complete $G$-universe always exists, one can be obtained by setting $U$ to be direct sum of
a representative of each isomorphism class of irreducible representations of $G$,
then defining $\mcU$ to be the direct sum of countably many copies of $U$.

\begin{remark}\label{rmk:vectorfixed}
Since a profinite group $G$ is compact and Hausdorff, the action of $G$ of a finite dimensional $G$-inner product space 
factors through a Lie group quotient by \cite[Lemma A.1]{fausk}. The only such quotients of a profinite group are 
finite, hence if $V$ is a finite dimensional $G$-inner product space, there is an open normal subgroup $N$ of $G$
such that $V^N=V$.  

In particular, the one-point compactification of $V$, denoted $S^V$ is fixed by some open normal subgroup $N$.
Thus, $S^V$ can be given the structure of a finite (pointed) $G/N$-CW-complex, by Illman \cite{illman}, 
and hence is a finite (pointed) $G$-CW-complex.  
\end{remark}

For brevity we define equivariant orthogonal spectra in terms of enriched functors
from a particular enriched category made using Thom spaces. 
Recall that $\eqtop{G}$ is enriched over itself, via the space of (not-necessarily equivariant)
maps, where $G$ acts by conjugation
\[
\big(f \colon X \lra Y \big) \mapsto \big( g_Y \circ f \circ g^{-1}_X \colon X \lra Y \big).
\]

\begin{definition}
Define an \emph{indexing space} to be a finite dimensional sub $G$-inner product space of $\mcU$.
We define $\mcL$ to be the category of all real $G$-inner product spaces that are isomorphic to indexing
G-spaces in $\mcU$ and morphisms the (not-necessarily equivariant) linear isometries. 
\end{definition}

\begin{definition}
For each $V \subseteq W$ there is a vector bundle (a subset of the product bundle)
\[
\gamma(V,W) = \{ (f,x) \subseteq \mcL(V,W) \times W  \mid x \in W-f(V) \}
\]
over $\mcL(V,W)$, where $W-f(V)$ is the orthogonal complement of $f(V)$ in $W$. 

Let $J(V,W)$ be the Thom space of $\gamma(V,W)$, with $G$-action given by 
$g(f,x)  = (gf g^{-1}, gx)$. 
\end{definition}

\begin{lemma}
For inclusions of indexing spaces $U \subseteq V \subseteq W$ the $G$-equivariant map
\[
\begin{array}{rcl}
\gamma(V,W) \times \gamma(U,V) & \lra & \gamma(U,W) \\
( (f,x), (k,y) ) & \longmapsto & (f \circ k, x+f(y) )
\end{array}
\]
induces a composition for the $\eqtop{G}$-enriched category $\mcJ$ whose objects are the objects of $\mcL$
and morphisms $G$-spaces are given by $\mcJ(V,W)$.
\end{lemma}

\begin{definition}
An \emph{orthogonal $G$-spectrum} $X$ on a universe $\mcU$ is an $\eqtop{G}$-enriched functor 
from $\mcJ$ to $\eqtop{G}$. A map of orthogonal $G$-spectra is a $\eqtop{G}$-enriched natural transformation.
The category of orthogonal $G$-spectra is denoted $\GspO$. 
\end{definition}

In particular, an orthogonal $G$-spectrum $X$ 
defines based $G$-spaces $X(V)$ for each indexing space $V \subset \mcU$ and based $G$-maps
\[
\sigma_{V,W} \colon S^{W-V} \smashprod X(V) \lra X(W)
\]
for each $V \subseteq W$. 
A map $f \co X \to Y$ of orthogonal $G$-spectra defines a set of 
$G$-maps $f \co X(V) \to Y(V)$, which commute with the maps $\sigma_{V,W}$
for each $V \subseteq W$. 

\subsection{Model categories of spectra}
With the category of $G$-spectra defined, the next step is to give model structures. 
This follows the usual path, a levelwise model structure, then a stable model structure and finally 
a rational model structure (as a Bousfield localisation of the stable model structure). 
In each case, the weak equivalences are defined using the \emph{open} subgroups of $G$. 
For brevity, we state the existence of the rational model structure as a theorem, 
giving the essential properties afterwards. These results are standard, 
details can be found in \cite[Section 2.2]{barnesthesis}.

\begin{definition}
Let $X$ be a $G$-spectrum and $n$ a non-negative integer. We define the $H$-\emph{homotopy groups}
of $X$ as 
\begin{align*}
\pi_n^H(X)  & = \colim_{V} \pi_n^H(\Omega^V X(V)) \\
\pi_{-n}^H(X)  & = \colim_{V \supset \bR^n} \pi_0^H(\Omega^{V-\bR^n} X(V)).
\end{align*}
In the first case the colimit runs over the indexing spaces of $\mcU$, in the second case 
over the indexing spaces of $\mcU$ that contain $\bR^n$. 

A map $f \co X \to Y$ of orthogonal $G$-spectra is called a \emph{$\pi_*$-isomorphism} if 
$\pi_k^H(f)$ is an isomorphism for all open subgroups $H$ of $G$ and all integers $k$. 
We call $f$ a \emph{rational $\pi_*$-isomorphism} if 
$\pi_k^H(f) \otimes \bQ$ is an isomorphism for all open subgroups $H$ of $G$ and all integers $k$.
\end{definition}

The rational model structure on $G$-spectra is made using 
the rational sphere spectrum $S^0 \bQ$,  
see \cite[Definition 1.5.2]{barnesthesis} for a construction as an
equivariant Moore spectrum for $\bQ$. 

\begin{theorem}
There is a cofibrantly generated proper stable model structure on the category of orthogonal $G$-spectra
whose weak equivalences are the class of rational $\pi_*$-isomorphisms and whose cofibrations are 
the class of $q$-cofibrations. This model structure is called the \emph{rational model structure}
and we write $\QGspO$ for the model category of orthogonal $G$-spectra equipped with the rational model structure.
\end{theorem}
\begin{proof}
The stable model structure on orthogonal $G$-spectra is cofibrantly generated, proper and stable, 
thus \cite[Theorem 7.2.17]{brfoundations} implies that 
any left Bousfield localisation of the stable model structure at a set of map closed under desuspension
is also cofibrantly generated, proper and stable. 

We construct the rational model structure by localising the stable model structure 
at the set of all suspensions and desuspensions of the map from the sphere spectrum 
to the rational sphere spectrum
\[
S^0 \longrightarrow S^0 \bQ.
\]
That this gives the correct weak equivalences is a consequence of Proposition \ref{prop:rationalhomotopygroups}.
\end{proof}

\begin{proposition}\label{prop:rationalhomotopygroups}
There is a natural isomorphism
\[
\pi_*^H (X \smashprod S^0 \bQ) \cong \pi_*^H (X) \otimes \bQ.
\]
A $G$-spectrum $X$ is fibrant in $\QGspO$ if and only if $X$ is an $\Omega$-spectrum 
and has rational homotopy groups.
\end{proposition}

The last statement gives a zig-zag of $\pi_*$-isomorphisms for any $G$-spectrum $X$:
\[
\hat{f}_\bQ X \lra \hat{f}_\bQ X \smashprod S^0 \bQ \longleftarrow X \smashprod S^0 \bQ.
\]
This shows that our localisation is a smashing localisation. 

\begin{corollary}\label{cor:rationalmapstensor}
If the $G$-spectrum $A$ is compact in the homotopy category of $\GspO$, then there is a natural isomorphism 
\[
[A,X]_\bQ \cong [A,X] \otimes \bQ.
\]
In particular, for $\hat{f}_\bQ X$ the fibrant replacement of $X$ in $\QGspO$, there is a
natural isomorphism
\[
\pi_*^H (\hat{f}_\bQ X) \cong \pi_*^H (X) \otimes \bQ.
\]
\end{corollary}

Since weak equivalence of the the stable model structure are defined in terms of $\pi_*$-isomorphisms,
the triangulated category $\Ho(\GspO)$ has a set of compact generators: the suspension spectra
$\Sigma^\infty G/H_+$, for $H$ an open subgroup of $G$. Similarly, this 
set is also a set of compact generators for $\Ho(\QGspO)$, as the weak equivalences of $\QGspO$
are defined in terms of rational $\pi_*$-isomorphisms. 

\begin{corollary}\label{cor:orbitgenerate}
For $G$ a profinite group, the homotopy category of $\QGspO$ is generated by the set of compact objects 
$\Sigma^\infty G/H_+$, for $H$ an open subgroup of $G$. 
\end{corollary}

This completes the construction of the model category we wish to model with algebra. 
Our next task is to study maps between objects like $\Sigma^\infty G/H_+$.  
We use the profinite version of tom-Dieck splitting to show these maps are concentrated in 
degree zero. 

\begin{proposition}[{\cite[Proposition 7.10]{fausk}}]\label{prop:tomdieck}
For $G$ a profinite group and $X$ a based $G$-space, there is an isomorphism of abelian groups
\[
\bigoplus_{(H) \opensub G} \pi_* 
\left( \Sigma_\infty EW_G H_+ \smashprod_{W_G H} X^H  \right)
\longrightarrow
\pi_*^G(\Sigma^\infty X).
\]
The sum runs over the conjugacy classes of open subgroups of $G$. 
\end{proposition}

\begin{theorem}[{\cite[Theorem 2.9]{barnespadic}}]\label{thm:homotopydegreezero}
For $G$ a profinite group, the graded $\bQ$-module 
\[
[\Sigma^\infty G/H_+, \Sigma^\infty G/K_+]^{G}_* \otimes \bQ
\]
is concentrated in degree zero.
\end{theorem}

% \begin{proof}
% We first change the problem to calculating stable homotopy groups 
% \[
% [\Sigma^\infty G/H_+, \Sigma^\infty G/K_+]^{G}_*  \cong [\Sigma^\infty S^0 ,\Sigma^\infty G/K_+]^H_* 
% \cong \pi_*^H(\Sigma^\infty G/K_+)
% \]
% We split this homotopy group into a sum over the collection of open subgroups of $H$.
% \[
% \pi_*^H(\Sigma^\infty G/K_+) \cong 
% \bigoplus_{(L) \opensub H}
% \pi_* \big( \Sigma^\infty E W_H L_+ \smashprod_{W_H L} (G/K_+)^L  \big)
% \]
% Each $L$ in the above has finite index in $H$, so $W_H L$ is a finite group. 
% The $G$-set $G/K$ is also finite, hence $(G/K)^L$ is a finite $W_H L$ set. 
% Thus we can decompose $(G/K)^L$ as a finite coproduct of terms $W_H L/M_i$, 
% for $I$ some finite indexing set. We then have a series of isomorphisms which 
% follow from the standard results: $X \smashprod_G G/H \cong X/H$, 
% $(EG)/H = BH$ and $\pi_* (\Sigma^\infty BH_+) \otimes \bQ \cong \bQ$, for a finite group $H$.
% \begin{align*}
% \pi_*\big(\Sigma^\infty E W_H L_+ \smashprod_{W_H L} (G/K_+)^L  \big) \otimes \bQ
% & \cong 
% \oplus_{i \in I} \pi_*\big(\Sigma^\infty E W_H L_+ \smashprod_{W_H L} (W_H L/M_i)_+ \big) \otimes \bQ \\
% & \cong 
% \oplus_{i \in I} \pi_*\big(\Sigma^\infty (B M_i)_+ \big) \otimes \bQ\\
% & \cong 
% \oplus_{i \in I} \bQ \qedhere
% \end{align*}
% \end{proof}

\section{Spans and the stable orbit category}\label{sec:spans}

The aim for this section is Theorem \ref{thm:spantoorbit}, 
which provides a combinatorial description 
of the stable orbit category for $G$.
That is, the full triangulated subcategory of $\Ho(\GspO)$ 
(the homotopy category of rational $G$-spectra) 
defined by the suspension spectra of finite pointed $G$-sets
is shown to be equivalent to the category of 
spans of finite $G$-sets. 
This result is well-known in the case of finite groups
but is seemingly new in the case of profinite groups. 
The method of proof is to relate the profinite case to the finite case 
by describing maps in $\Ho(\GspO)$
as a colimit of maps in $\Ho(G/N \spO)$, as $N$ varies over 
the open normal subgroups of $G$, see Lemma \ref{lem:orbitfinitetoprofinite}.

\subsection{The Burnside category}

In the case of a finite group $G$, the Burnside category is the category of $G$-sets with 
(equivalence classes) of spans as morphisms. In this subsection we generalise this 
construction to profinite groups and show how it relates to the finite group case.

\begin{definition}
A set $X$ with an action of $G$ is said to be \emph{discrete} if 
the canonical map 
\[
\colim_{H \opensub G} X^H \longrightarrow X
\]
is an isomorphism. The category of finite discrete  $G$-sets and 
equivariant maps is denoted $\Gsetsdiscfin$.
\end{definition}
Equally, one can define a discrete $G$-set as a $G$-set such that the
stabiliser of each point is open. We then see that a $G$-set $X$ is discrete 
if and only if the action on $G$ is continuous, when $X$ is equipped with the discrete topology. 
If $X$ is a finite discrete $G$-set, then the stabiliser of each point is open, 
as is the intersection of all the stabilisers. Thus a finite $G$-set $A$ is discrete
if and only if there is an open subgroup $H$ of $G$ such that $A$ is $H$-fixed. 

The class of discrete $G$-sets is closed under arbitrary coproducts and finite products.
The set of finite coproducts of $G$-sets of the form $G/H$, for $H$ an open subgroup of $G$, 
is a skeleton for the class of finite discrete  $G$-sets. 
We also note that the empty set is a finite discrete  $G$-set. 

\begin{definition}
The \emph{Burnside ring} of $G$, written as $\intburnsidering{G}$ is the Grothendieck ring of finite discrete $G$-sets. 
We further define the \emph{rational Burnside ring} of $G$ as $\ratburnsidering{G} = \intburnsidering{G} \otimes \bQ$.
\end{definition}

\begin{lemma}
For $G$ a profinite group there is a natural isomorphism 
\[
\varepsilon^* \co \underset{N \opennormalsub G}{\colim}\intburnsidering{G/N} \lra \intburnsidering{G}.
\]
\end{lemma}
\begin{proof}
A finite $G/N$-set $A$ can be considered as a $G$-set via inflation, $\varepsilon^* A$. 
Since each point of $\varepsilon^* A$ is fixed by $N$, this is a finite discrete $G$-set.
Inflation gives an injective ring map
\[
\varepsilon^*_N \co \intburnsidering{G/N} \lra \intburnsidering{G}.
\]
The maps $\varepsilon^*_N$ are compatible with the maps forming the colimit (as these are also inflation maps), 
giving the map $\varepsilon^*$. 
As the colimit is filtered, $\varepsilon^*$ is also injective. 

Any finite $G$-set $A$ is fixed by some open subgroup $H$, which must contain an open 
normal subgroup $N$. Thus $A$ is in the image of $\varepsilon^*_N$ and so 
$\varepsilon^*$ is surjective.
\end{proof}

We can generalise the construction of the Burnside ring to make a category. 

\begin{definition}
Let $G$ be a profinite group. 
A \emph{span} of finite discrete $G$-sets, 
a pair of equivariant maps $B \xleftarrow{\beta} A \xrightarrow{\gamma} C$, sometimes shortened to $(\beta, \gamma)$.
Two spans $B \xleftarrow{\beta} A \xrightarrow{\gamma} C$ and 
$B \xleftarrow{\beta'} A' \xrightarrow{\gamma'} C$
are \emph{equivalent} if there is an equivariant isomorphism $A \to A'$
such that the following diagram commutes. 
\[
\xymatrix@R-0.5cm@C+0.5cm{
& \ar[dl]_{\beta} A \ar[dr]^{\gamma} \ar[dd]_\alpha\\
B 
&& C \\
& \ar[ul]^{\beta'} A' \ar[ur]_{\gamma'}
}
\]
We write $[\beta, \gamma]$ for the equivalence class of $(\beta, \gamma)$.
\end{definition}

We recall the notion of composition of spans.
Take two spans
$B \xleftarrow{\beta} A \xrightarrow{\gamma} C$ and 
$C \xleftarrow{\gamma'} A' \xrightarrow{\delta} D$, then construct $A''$ as the pullback of $\gamma$ and $\gamma'$

\[
\xymatrix@R-0.5cm@C+0.5cm{
&&  \ar[dl]_{\theta} A'' \ar[dr]^{\theta'}
\\
& \ar[dl]_{\beta} A \ar[dr]^{\gamma} 
&& \ar[dl]_{\gamma'} A' \ar[dr]^{\delta} 
\\
B 
&& C 
&&
D
}
\]
The composite of $(\beta, \gamma)$ and $(\gamma', \delta)$ is the span
$(\beta \circ \theta, \delta \circ \theta')$. 
This composition is well-defined under equivalence of spans. 

Further, there is an addition on (equivalence classes of) spans.
Consider two spans $B \xleftarrow{\beta} A \xrightarrow{\gamma} C$ and 
$B \xleftarrow{\beta'} A' \xrightarrow{\gamma'} C$. 
Their sum is the span 
\[
B \xleftarrow{\beta, \beta'} A \coprod A' \xrightarrow{\gamma, \gamma'} C.
\]
Additions of spans is associative, compatible with equivalence of spans
and the unit is the span $\emptyset \leftarrow \emptyset \rightarrow \emptyset$.

\begin{definition}
The \emph{Burnside category} of $G$ is the category 
with objects the finite discrete $G$-sets and 
morphisms given by Grothendieck construction on sets of 
equivalence classes of spans of finite discrete $G$-sets, denoted $\spanunstableorbcat$. 
\end{definition}

Thus, a map $A \to B$ in the Burnside category is a formal difference of equivalence classes of spans.
We also see that the Burnside ring of $G$ is the ring $\spanunstableorbcat(\ast, \ast)$. 

Just as for the Burnside ring, we can relate the Burnside category 
for $G$ to the Burnside categories for the groups $G/N$, 
where $N$ is an open normal subgroup of $G$. 

\begin{lemma}\label{lem:spanstofinite}
If $A$ and $B$ are finite discrete $G$-sets, for $G$ a profinite group, then there is an isomorphism:
\begin{align*}
\underset{N \opennormalsub G}{\colim} \spanunstableorbcatN{G/N}\left(A,B \right)
\cong 
\spanunstableorbcat \left(A,B \right).
\end{align*}
\end{lemma}
\begin{proof}
The maps of the colimit are given by inflation functors and inflation from $G/N$ to $G$
induces the map to the codomain. 
For the inverse, take a span
\begin{align*}
\xymatrix{B  &  A \ar[r]  \ar[l]  & C},
\end{align*}
and choose an open normal subgroup $N$ that fixes each of element of $A$, $B$ and $C$.  
Then this span appears in term $N$ of the colimit.
\end{proof}

\subsection{The stable orbit category}
In this subsection we study maps in the $G$-equivariant stable homotopy category  
between spectra of the form $\Sigma^\infty A_+$, where $A$ is a finite discrete $G$-set. 
Our aim is to relate this to unstable homotopy classes of maps
of $G/N$-spaces, where $N$ runs over the open normal subgroups of $G$. 
We start by comparing maps in the $G$-equivariant stable homotopy category 
to the $G$-equivariant unstable homotopy category.

We need two categories, one defined via the $G$-equivariant stable category 
and one via the rational analogue. It is important to note that these categories are 
not graded, we use $\pi_0$ in the notation of the categories as a reminder of this fact. 
The last sentence of the definition holds due to Corollary \ref{cor:rationalmapstensor}.

\begin{definition}\label{def:orbitcats}
We define a category $\piorbcat$, called the \emph{$G$-equivariant stable orbit category}.
The objects are the class of $G$-spectra of the form $\Sigma^\infty A_+$, 
for $A$ a finite discrete $G$-set.
The morphisms are given by 
\[
\piorbcat(\Sigma^\infty A_+, \Sigma^\infty B_+) = [\Sigma^\infty A_+, \Sigma^\infty B_+]^G
\]
the set of maps in the homotopy category of $\GspO$.

Similarly, we define a category $\piratorbcat$, called the \emph{rational $G$-equivariant stable orbit category}.
The objects are the same as for $\piorbcat$, but the morphisms are given by 
\[
\piratorbcat(\Sigma^\infty A_+, \Sigma^\infty B_+) = [\Sigma^\infty A_+, \Sigma^\infty B_+]^G \otimes \bQ
\]
the set of maps in the homotopy category of $\QGspO$.
\end{definition}

Recall that a finite discrete $G$-set is a disjoint union of homogeneous spaces $G/H$ 
for $H$ open, hence every finite discrete $G$-set $A$ is a finite $G$-CW-complex
and $A_+$ is finite pointed $G$-CW-complex. 

\begin{lemma}[{\cite[Corollary 7.2]{fausk}}]\label{lem:stabletounstable}
For $G$ a profinite group, there is an isomorphism of abelian groups
\[
[\Sigma^\infty A, \Sigma^\infty B]^G
\cong 
\colim_{W \in \mcU} [A \smashprod S^W, B \smashprod S^W]^{G}
\]
where the right hand terms indicate maps in the homotopy category of pointed $G$-spaces
and $A$ and $B$ are finite pointed $G$-CW-complexes. 
\end{lemma}

Now we show how the equivariant stable homotopy category for $G$ relates to the 
equivariant stable homotopy category for finite quotients $G/N$. 
Whenever we talk about $G/N$-spectra, we will use the $G/N$-universe 
$\mcU^N$. 
This universe is a complete $G/N$-universe as any finite-dimensional 
$G/N$-inner product space $V$ can be written as $(\varepsilon^* V)^N$
and $\varepsilon^* V$ is isomorphic to an indexing space of $\mcU$, as 
$\mcU$ is complete.

Let $A$ and $B$ be finite pointed $G$-CW-complexes, and $N_1 \leqslant N_2$ be open normal subgroups of $G$, which
fix all of $A$ and $B$.  
Then $A$ can be considered as either a $G/N_2$-CW-complex or a $G/N_1$-CW-complex and the inflation functor 
$\varepsilon^*$ from $G/N_2$-spaces to $G/N_1$-spaces sends the $G/N_2$-version of $A$ to the $G/N_1$-version.
Hence, the inflation functor $\varepsilon^*$ induces a natural map
\[
[\Sigma^\infty A, \Sigma^\infty B]^{G/N_2} \lra [\Sigma^\infty A, \Sigma^\infty B]^{G/N_1}.
\]

We can find a more direct description of this map using Lemma \ref{lem:stabletounstable}. 
An element of the domain can be represented as a map of pointed $G/N_2$-spaces
\[
f\colon A\smashprod S^{V}\rightarrow B\smashprod S^{V}
\]
for some indexing space $V \subset \mcU^{N_2}$. 
Applying the inflation functor gives a map of pointed $G/N_1$-spaces $\varepsilon^* f$
and since $V$ is $N_2$-fixed, this defines a element of 
\[
\colim_{W \in \mcU^{N_1}} [A \smashprod S^W, B \smashprod S^W]^{G/N_1}. 
\]
This assignment is compatible with taking colimits over $V^{N_2} \subset \mcU^{N_2}$
and taking $G/N_2$-equivariant homotopy classes. Hence, we may construct
the colimit term of the following result.

\begin{lemma}\label{lem:orbitfinitetoprofinite}

If $A$ and $B$ are finite pointed $G$-CW-complexes, for $G$ a profinite group, 
then there is an isomorphism of abelian groups:
\begin{align*}
\underset{N \opennormalsub G}{\colim}
\left[\Sigma^\infty A, \Sigma^\infty  B\right]^{G/N} 
\cong \left[\Sigma^\infty A, \Sigma^\infty  B\right]^G.
\end{align*}
\end{lemma}
\begin{proof}
We have seen that there are isomorphisms 
\[
\underset{N \opennormalsub G}{\colim}
\left[\Sigma^\infty A, \Sigma^\infty B\right]^{G/N} 
\cong
\underset{N \opennormalsub G}{\colim}
\left(
\colim_{W \in \mcU^N} [A \smashprod S^W, B \smashprod S^W]^{G/N}
\right)
\]
The right hand term maps to 
\[
\colim_{W \in \mcU} [A \smashprod S^W, B \smashprod S^W]^{G} 
\cong 
\left[\Sigma^\infty A , \Sigma^\infty B \right]^G
\]
via the inflation functors. 

We construct an inverse. Take a representative 
$f\colon A\wedge S^V\rightarrow B\wedge S^V$. 
The finite pointed $G$-CW-complexes $A$ and $B$ are fixed by some open normal subgroup.
By Remark \ref{rmk:vectorfixed}, the $G$-inner product space $V$ 
must also be fixed by some open normal subgroup of $G$. 
By taking intersections there is an open normal
subgroup $N$ of $G$, which also fixes all of $A$, $B$ and $V$. 
It follows that $f$ defines an element of 
$[A \smashprod S^V, B \smashprod S^V]^{G/N}$.
One can verify that this is an inverse to the map of the statement.
\end{proof}

Recall the grading convention $[\Sigma^n X, \Sigma^m Y]^G = [X,Y]_{n-m}^G$.
Since suspension by $S^1$ preserves finite (pointed) $G$-CW-complexes, we have the following
extension. 
\begin{corollary}
For $G$ a profinite group, there is an isomorphism of graded abelian groups
\begin{align*}
\underset{N \opennormalsub G}{\colim}
\left[\Sigma^\infty A, \Sigma^\infty  B\right]_*^{G/N} 
\cong \left[\Sigma^\infty A, \Sigma^\infty  B\right]_*^G
\end{align*}
for $A$ and $B$ finite $G$-CW-complexes. 
\end{corollary}

\begin{theorem}\label{thm:spantoorbit}
For $G$ a profinite group, there is an equivalence of categories
\begin{align*}
\psi_G \colon
\spanunstableorbcat 
& \longrightarrow 
\piorbcat
\\
A&\mapsto \Sigma^{\infty}A_{+}\,\text{on objects}
\\
\left[ B \xleftarrow{\beta} A \xrightarrow{\alpha} C  \right] 
&
\mapsto \Sigma^{\infty}\beta \circ \tau(\alpha)\,\text{on morphisms}
\end{align*}
where $\tau(\alpha)$ is the transfer map construction associated to $\alpha$, see 
\cite[Construction II.5.1]{lms86} or work of the second author \cite[Construction 3.1.11]{sugruethesis}.
\end{theorem} 
\begin{proof}
That $\psi_G$ is an equivalence for finite groups $G$ is well-known, see \cite[Proposition V.9.6]{lms86}.
We will use that result to extend from the finite case to the profinite case. 

For an inclusion $N \to N'$ of open normal subgroups, the inflation functor 
from $\spanunstableorbcatN{G/N}$ to  $\spanunstableorbcatN{G/N'}$ 
commutes with $\psi_{G/N}$ and $\psi_{G/N'}$, hence the following diagram commutes.
\begin{align*}
\xymatrix@C+1.5cm{
\spanunstableorbcat \left( A,B \right)
  \ar[r]^{\psi_G} &
\left[\Sigma^\infty A_{+},\Sigma^\infty B_{+}\right]^G
\\  
\underset{N \opennormalsub G}{\colim} \spanunstableorbcatN{G/N}\left(A,B\right)
\ar[r]_-{\cong}^{\psi_{G/N}}  \ar[u]^{\cong} &
\underset{N \opennormalsub G}{\colim} \left[\Sigma^\infty A_{+}, \Sigma^\infty B_{+}\right]^{G/N}
\ar[u]^{\cong}
}
\end{align*}
This proves that $\psi_G$ is full and faithful, essential surjectivity is immediate. 
\end{proof}

\section{The classification}

We give the main result, the classification (in terms of Quillen equivalences) 
of rational $G$-equivariant spectra, for profinite $G$, 
in terms of a simple algebraic model. In fact, by previous work of the authors \cite{BSmackey}
we give two equivalent algebraic models. The first
is the category of (chain complexes of) rational $G$-Mackey functors, the second is the category 
of Weyl $G$-sheaves over the space of closed subgroups of $G$. The relative advantages of two
descriptions are explained in that reference, though we will need the sheaf description in  
Section \ref{sec:injdim}.

\subsection{Mackey functors}\label{subsec:mackey}

There are several equivalent definitions of Mackey functors, we briefly describe 
the three most common variations, leaving the axioms of the first two versions for references such as 
Lindner \cite{lindner} or Thiel \cite{Thiel}.  
These three definitions are shown to be equivalent in 
\cite[Section 2.1]{sugruethesis}, which follows the work of \cite{lindner}.

\begin{enumerate}
\item A set of abelian groups $M(G/H)$, for $H$ open in $G$, with induction, restriction and conjugation maps 
relating these groups that satisfy a list of axioms (unital, transitivity, associativity, equivariance and the Mackey axiom).
\item A pair of functors from the category of finite discrete $G$-sets to abelian groups, one covariant, 
one contravariant that agree on objects and satisfy a pullback axiom and a coproduct axiom. 
These are sometimes known as categorical Mackey functors.
\item An additive functor from the Burnside category $\spanunstableorbcat$ to abelian groups.
\end{enumerate}

The choice of focusing on the open subgroups (or equally the discrete finite $G$-sets)
matches with the `finite natural Mackey system'' of Bley and Boltje
\cite[Definition 2.1 and Examples 2.2]{BB04} and Thiel \cite[Definition 2.2.12]{Thiel}.
In the case of a finite group, this choice restricts to the usual definitions.
For our purposes we use the last definition.

\begin{definition}
A \emph{Mackey functor} for a profinite group $G$ is an additive functor 
from the Burnside category $\spanunstableorbcat$ to abelian groups. 
We will write $\intmackey{G}$ for the category of Mackey functors and 
additive natural transformations between them. 

A \emph{rational Mackey functor} is an additive functor 
from the Burnside category $\spanunstableorbcat$ to $\bQ$-modules. 
We will write $\mackey{G}$ for the category of rational Mackey functors and 
additive natural transformations between them. 
\end{definition}

General examples of Mackey functors can be found in the references
given at the start of the section. However, there is one class
of rational Mackey functors of particular relevance to this paper. 

\begin{example}\label{ex:homotopygroupmackey}
By Theorem \ref{thm:spantoorbit} and Corollary \ref{cor:rationalmapstensor}, 
if $X$ is a $G$-spectrum then we have a rational $G$-Mackey functor:
\begin{align*}
\piratorbcat &
\longrightarrow \mathbb{Q} \leftmod\\G/H_{+} & 
\longmapsto \pi_0^H\left(X\right) \otimes \bQ \cong [G/H_+,X]^G_{\bQ}
\end{align*}
called the \emph{homotopy group Mackey functor} of $X$.
\end{example}

We write out the preceding definition in terms of notation and apply Theorem \ref{thm:spantoorbit}.
\begin{align*}
\intmackey{G} & = \func_{\abgps}\left(\spanunstableorbcat, \abgps \right)
\cong 
\func_{\abgps}\left(\piorbcat, \abgps \right) \\
\mackey{G} & = \func_{\abgps}\left(\spanunstableorbcat, \bQ \leftmod \right)
\cong 
\func_{\abgps}\left(\piorbcat, \bQ \leftmod \right)
\end{align*}

It remains to relate these category of $G$-Mackey functors to the model category of 
rational $G$-spectra. For this we will need a model structure on chain complexes 
of $G$-Mackey functors.

\begin{lemma}
There is a cofibrantly generated model structure on the category of chain complexes of (rational) $G$-Mackey functors
where a map is a fibration if and only if it is a surjection and the class of weak equivalences is 
the class of homology isomorphisms.
\end{lemma}
\begin{proof}
Since $\Ch(\intmackey{G}) = \func_{\abgps}\left(\spanunstableorbcat, \Ch(\bZ) \right)$
we use the cofibrant generation of $\Ch(\bZ)$ to obtain generating sets for 
$\Ch(\intmackey{G})$ in terms of the representable functors. 
\end{proof}

\subsection{Tilting theory}\label{subsec:tilting}

We give our classification theorem for rational $G$-spectra, where $G$ is a profinite group.

\begin{theorem}\label{thm:gpsectramackey}
For $G$ a profinite group, there is a zig-zag of Quillen equivalences between the model category of rational $G$-spectra
and the model category of chain complexes of rational $G$-Mackey functors. 
\end{theorem}
\begin{proof}
Choose a skeleton $\mcG$ of $\piratorbcat$ (such as the set of finite coproducts of the $G$-sets 
$G/H$ for $H$ an open subgroup of $G$). 
Define $\pi_0(\mcG)$ to be the category whose objects are the elements of $\mcG$
and whose morphisms are given by the abelian groups
\[
[\Sigma^\infty A_+, \Sigma^\infty B_+]^G \otimes \bQ.
\]
The objects of $\mcG$ are a set of compact generators for $\QGspO$ by Corollary \ref{cor:orbitgenerate}
and the set of graded maps between them in $\Ho(\QGspO)$ is 
concentrated in degree zero by Theorem \ref{thm:homotopydegreezero}. 
Thus, we can use \cite[Theorem 5.1.1 and Proposition B.2.1]{ss03stabmodcat} of Schwede and Shipley 
to obtain a zig-zag of Quillen equivalences
\[
\QGspO \simeq 
\Ch\left( \func_{\abgps}\left(\pi_0(\mcG), \abgps \right)  \right).
\]
As $\pi_0(\mcG)$ is a skeleton of $\piratorbcat$, we have the first equivalence of categories below.
The second is Lemma \ref{lem:changeofQ}, which applies 
as Corollary \ref{cor:rationalmapstensor} shows that $\piratorbcat = \piorbcat \otimes \bQ$.
\[
\Ch\left( \func_{\abgps}\left(\pi_0(\mcG), \abgps \right)  \right)
\cong 
\Ch\left( \func_{\abgps}\left(\piratorbcat, \abgps \right)  \right)
\cong 
\Ch\left( \func_{\abgps}\left(\piorbcat, \bQ \leftmod \right)  \right)
\]
Applying Theorem \ref{thm:spantoorbit} gives the final step
\begin{align*}
\Ch(\mackey{G}) & = \Ch(\func_{\abgps}\left(\spanunstableorbcat, \bQ \leftmod \right))
\cong 
\Ch(\func_{\abgps}\left(\piorbcat, \bQ \leftmod \right)). \qedhere
\end{align*}
\end{proof}

\begin{lemma}\label{lem:changeofQ}
Let $\mcC$ be a small additive category. 
The rationalisation functor $i \co \mcC \to \mcC \otimes \bQ$ (defined in the proof)
and the forgetful functor $U \co \bQ \leftmod \to \text{Ab}$ induce equivalences
\[
\xymatrix{
\func_{\abgps}\left(\mcC, \bQ \leftmod \right)
&
\func_{\abgps}\left(\mcC \otimes \bQ, \bQ \leftmod \right)
\ar[l]_{i_*} \ar[r]^{U^*} 
&
\func_{\abgps}\left(\mcC \otimes \bQ, \abgps \right).
} 
\] 
\end{lemma}
\begin{proof}
A small additive category $\mcC$ has a rationalisation $\mcC \otimes \bQ$, 
this category has the same objects and morphisms given by 
\[
(\mcC \otimes \bQ)(c,c') = \mcC(c,c') \otimes \bQ.
\]
Composition is induced from that of $\mcC$ and  $\mcC \otimes \bQ$ is an additive category.
Moreover, there is an additive functor from $\mcC$ to its rationalisation 
$i \co \mcC \to \mcC \otimes \bQ$. 

The functors are equivalences as in each case the functor must take values in $\bQ$-modules. 
\end{proof}

\begin{example}
Let $A$ be an finite discrete $G$-set. 
Let $M_A$ the homotopy group Mackey functor 
of $\Sigma^\infty A_+$ from Example \ref{ex:homotopygroupmackey}.
Then $M_A$ is the representable functor given by 
\[
\spanunstableorbcat ( -, A) \otimes \bQ.
\]
\end{example}

We recall the notion of Weyl-$G$-sheaves over the space of closed subgroups of $G$
from \cite[Section 2]{BSmackey} and \cite[Section 10]{BSsheaves}.
For $G$ a profinite group, let $\sub G$ denote the set of closed subgroups of $G$.
We topologise this set as the limit of finite discrete spaces 
\[
\sub G \coloneqq \underset{N \opennormalsub G}{\lim}  \sub (G/N)
\]
using the maps which send $K \in \sub G$ to $KN/N \in \sub (G/N)$.  

\begin{definition}\label{defn:eqsheafRmod}
A \emph{$G$-equivariant sheaf} of $\bQ$-modules over $\sub G$ is a map
$p \co E \to \sub G$ such that:
\begin{enumerate}
\item \label{item:sheafeq} $p$ is a $G$-equivariant map $p \colon E\lra \sub G$ of spaces with continuous $G$-actions,
\item \label{item:sheafab} $(E,p)$ is a sheaf space of $\bQ$-modules,
\item \label{item:sheafcomb} each map $g \co p^{-1} (x) \lra p^{-1} (g x)$ 
is a map of $\bQ$-modules for every $x\in \sub G,$ $g\in G$.
\end{enumerate}
We will write this as either the pair $(E,p)$ or simply as $E$.
A map $f \co (E,p) \to (E',p')$ of $G$-sheaves of $\bQ$-modules over $\sub G$ is a 
$G$-equivariant map $f \co E \to E'$ such that $p'f=p$ 
and $f_x \co E_x \to E'_x$ is a map of $\bQ$-modules for each $x \in \sub G$.  
\end{definition}

\begin{definition}\label{defn:weylsheaf}
A \emph{rational Weyl-$G$-sheaf} $E$ is a $G$-sheaf of $\bQ$-modules
over $\sub G$ such that $E_K$ is $K$-fixed and hence 
is a discrete $\bQ[W_G(K)]$-module. 
A map of Weyl-$G$-sheaves is a map of $G$-sheaves of $\bQ$ modules over $\sub G$.
We write this category as $\Rweylsheaf{G}{\bQ}$ 
\end{definition}

\begin{theorem}[{\cite[Theorem A]{BSmackey}}]\label{thm:equivalencemain}
If $G$ is a profinite group then the category of rational $G$-Mackey functors is equivalent to the 
category of rational Weyl-$G$-sheaves over $\sub G$. Furthermore, this is an exact equivalence.
\end{theorem}

\begin{corollary}\label{cor:gpsectrasheaf}
For $G$ a profinite group, there is a zig-zag of Quillen equivalences between the model category of rational $G$-spectra
and the model category of chain complexes of rational Weyl-$G$-sheaves. 
\end{corollary}

\section{Homological dimension of equivariant sheaves}\label{sec:injdim}
Using the Weyl-$G$-sheaf description of rational $G$-spectra we can calculate the 
homological dimension (also known as the injective dimension) of the algebraic model. 
This gives an indication of the homological 
complexity of the algebraic model. It is already known that the algebraic model 
in the case of finite groups has homological dimension zero (we will recover this result).
The algebraic model in the case of an $r$-torus $(S^1)^{\times r}$ is $r$, as shown in 
Greenlees \cite[Theorem 8.1]{tnq2}. 

We prove that the Cantor--Bendixson rank (Definition \ref{def:CBrank}) of the space $\sub G$
will determine the homological dimension of rational (Weyl) $G$-sheaves on $\sub G$.
This is an equivariant generalisation of the results of the second author, \cite{sugrue19}. 

\subsection{Cantor--Bendixson rank}

We start with the basic definitions, see Gartside and Smith, \cite{GartSmith10} and \cite{GartSmith10count}.
\begin{definition}
For a topological space $X$ we define the \emph{Cantor--Bendixson process} on $X$. 
Denote by $X^{\prime}$ the set of all isolated points of $X$. 
\begin{enumerate}
\item Let $X^{(0)}=X$ and $X^{(1)}=X\setminus X^{\prime}$ have the subspace topology with respect to $X$.
\item For successor ordinals suppose we have $X^{(\alpha)}$ for an ordinal $\alpha$, we define 
\[
X^{(\alpha+1)}=X^{(\alpha)}\setminus {X^{(\alpha)}}^{\prime}.
\]
\item If $\lambda$ is a limit ordinal we define 
\[
X^{(\lambda)}=\underset{\alpha<\lambda}{\bigcap}X^{(\alpha)}.
\]
\end{enumerate}
\end{definition}
Every Hausdorff topological space $X$ has a minimal ordinal $\alpha$ such that $X^{(\alpha)}=X^{(\lambda)}$ 
for all $\lambda\geq \alpha$, see \cite[Lemma 2.7]{GartSmith10}.

\begin{definition}
For $X$ a Hausdorff topological space, the \emph{Cantor--Bendixson rank} of $X$, 
written $\rank_{CB}(X)$, is the minimal ordinal $\alpha$ 
such that $X^{(\alpha)}=X^{(\lambda)}$ for all $\lambda\geq \alpha$.

A topological space $X$ is called \textbf{perfect} if it has no isolated points, whereupon 
$\rank_{CB}(X) = 0$.
\end{definition}

\begin{remark}
The definition given above agrees with that of Gartside and Smith, \cite{GartSmith10count}.
The convention of Dickmann, Schwartz and Tressl, \cite[Definition 4.3.1]{DSTspectral}, 
is to take one less than the rank as defined above.
\end{remark}

There are two ways that the Cantor--Bendixson process can stabilise, by reaching the empty set 
or a perfect subspace. 

\begin{definition}\label{def:CBrank}
If $X$ is a Hausdorff space with Cantor--Bendixson rank $\lambda$, then we define 
the \emph{perfect hull} of $X$ to be the subspace $X^{(\lambda)}$, written $X_H$.

We write $X_S$ for the complement $X \setminus X_H$ and call it the 
\emph{scattered part} of $X$.
The space $X$ is said to be \emph{scattered} if $X_H= \emptyset$. 
\end{definition}

\begin{example}\label{ex:rankexamples}
The Cantor--Bendixson rank of the empty set is zero
and the Cantor--Bendixson rank of a non-empty discrete space is 1 as every point
is isolated.

The space $\sub \adic{p}$ of closed subgroups of the $p$-adics is the subspace of $\bR$ consisting of the points
\[
\left\lbrace 1/n \mid n\in\mathbb{N}\right\rbrace\bigcup\left\lbrace 0\right\rbrace.
\] 
The isolated points are those of the form $1/n$, which are removed in the first stage of 
the Cantor--Bendixson process, thus $(\sub \adic{p})^{(1)} = \{ 0\}$.
For $k>1$, $(\sub \adic{p})^{(k)} = \emptyset$, so $\rank_{CB}(\sub \adic{p})=2$.
\end{example}

\begin{definition}
If $X$ is a space and $x\in X_S$, we define the \emph{height} of $x$ 
denoted $\text{ht}(X,x)$, to be the  ordinal $\kappa$ such that $x\in X^{(\kappa)}$ 
but $x\notin X^{(\kappa+1)}$. 
We denote this by $\text{ht}(x)$ when the background space $X$ is understood.
\end{definition}

We may rephrase the definitions to see that a point $x$ of height $k>0$ is 
a limit of points of height $k-1$. Consequently, an open neighbourhood of $x$
contains infinitely many points of height $j$ for each $j < k$. 

% \dave{remove next proposition if not needed}
% \begin{proposition}
% If $X$ is a Hausdorff space, then $X_H$ is always closed and $X_S$ is always open.
% \end{proposition}

\begin{lemma}\label{lem:equivariantcantor}
Let $X$ be a topological space with an action of a topological group $G$. 
The height of points of $X$ is invariant under the action of $G$.

If $x \in X$ has height $k>0$, then every neighbourhood of $x$ 
contains infinitely many points from orbits other than $Gx$. 
\end{lemma}
\begin{proof}
As $G$ acts through homeomorphisms, the first statement holds. 
The second statement follows from our preceding discussion and the first statement. 
\end{proof}

\subsection{Equivariant Godement resolutions}

We recap equivariant Godement resolutions from \cite[Section 9]{BSsheaves}.
The key change from the non-equivariant case is the use orbits in place of points.

\begin{definition}\label{def:godestage}
Let $p \colon E \to X$ be a $G$-equivariant sheaf of $\bQ$-modules 
over a $G$-space $X$, for $G$ a profinite group.
Define a $G$-sheaf 
\[
I^0(E) = \prod_{A} i^*_A E_{|A}
\]
with the product taken over the $G$-orbits of $X$ and $i^*_A$ the extension by zero functor 
induced by the map $A \to X$. 

The restriction--extension adjunction induces morphisms $E \lra i^*_A E_{|A}$, 
which combine to a monomorphism
\[
\delta_E \colon E \lra I^0(E).
\]
\end{definition}

\begin{lemma}
Let $p \colon E \to X$ be a $G$-equivariant sheaf of $\bQ$-modules 
over a $G$-space $X$, for $G$ a profinite group.
The $G$-sheaf $I^0(E)$ is injective.
\end{lemma}
\begin{proof}
We prove that $I^0(E)$ is the product of injective sheaves. 
For each orbit $A$ of $X$, pick an element $x_A \in A$. 
By \cite[Lemma 8.3]{BSsheaves} there is an isomorphism of sheaves over $A$ 
\begin{align*}
G\underset{\stab_G(x_A)}{\times}E_{x_A} \cong E_{|_{A}}.
\end{align*}
The left hand sheaf is known as the equivariant 
skyscraper sheaf of $E_{x_0}$ at $x_0$. It can be viewed as part of an adjunction
with the left adjoint being taking the stalk at $x_0$. As this left adjoint preserves 
monomorphisms, the equivariant skyscraper sheaf construction preserves injective objects. 

The result is completed by Castellano and Weigel \cite[Proposition 3.1]{CW16}, which  
states that every object of the category of discrete $\bQ[\stab_G(x_0)]$-modules
is injective. 
\end{proof}

\begin{remark}
The result \cite[Proposition 3.1]{CW16} also fixes an omission in work of 
the first author, \cite[Lemma 6.2]{barnespadic}, which implicitly assumes that 
all discrete $\bQ\left[\adic{p}\right]$-modules are injective.
\end{remark}

Iterating this construction $I^0$ gives an injective  resolution.

\begin{definition}\label{godeequi}\index{equivariant Godement resolution}
Let $G$ be a profinite group. If $E$ is a $G$-sheaf of $R$-modules over a profinite $G$-space we define the 
\emph{equivariant Godement resolution} as follows.
\begin{align*}
\xymatrix{
0\ar[r]
&
E\ar[r]^-{\delta_E}
&
I^0(E) \ar[d]^p \ar[r]
&
I^0(\coker\delta_E)=I^1(E)\ar[r]
&
\ldots\\&&
\coker\delta_E \ar[ur]_{\delta_{\coker\delta_E}}
}
\end{align*}
\end{definition}

We now connect the Cantor--Bendixson rank of a $G$-space $X$ to the length of the 
equivariant Godement resolution of rational $G$-sheaves over $X$.

\begin{theorem}
Let $E$ be a rational $G$-sheaf over a profinite $G$-space $X$, for $G$ a profinite group.
For $n \in \bN$ and $x \in X$, the stalk $I^n(E)_x$ is zero unless the height of $x$ is at least $n$. 
\end{theorem}
\begin{proof}
The proof is by induction. The base case follows from the fact that $(\delta_E)_x$ 
is an isomorphism for any isolated point of $X$.
Similarly, if $x$ has a open neighbourhood where the only non-trivial stalk is at $x$, 
then $(\delta_E)_x$ is an isomorphism. The equivariant Godemont resolution 
with the Cantor--Bendixson process gives the inductive step using Lemma \ref{lem:equivariantcantor}.
Further details are given in the non-equivariant case of \cite[Lemma 3.7]{sugrue19}. 
\end{proof}

% Let $x$ be an isolated point of $X$. Then $U=\{ x \}$ is an open set and $\delta_E$
% induces an isomorphism from $E_x=E(U)$ to 
% \[
% \prod_{A} i^*_A E_{|A}(U)  = 
% \prod_{A} \left( i^*_A E_{|A}(U) ) = 
% E_x.
% \]
% It follows that $0=(\coker \delta_0)_x= I^1(E)(U) = I^1(E)_x$, so that $I^1(E)$
% is zero on points of height 1. 

If we restrict ourselves to the case of scattered spaces of finite rank
(the Cantor--Bendixson process ends in the empty set after finitely many steps) 
this theorem gives an upper bound for the length of the equivariant Godemont resolutions. 

\begin{corollary}\label{cor:idupperbound}
Let $G$ be a profinite group. Let $X$ be a scattered profinite $G$-space of Cantor--Bendixson rank $n \in \bN$. 
The category of rational $G$-sheaves over $X$ has homological dimension at most $n-1$. 
\end{corollary}
\begin{proof}
If the rank is $n$, then $X^{(n)} =\emptyset$ and every point has height at most $n-1$.
Hence, every stalk of $I^n(X)$ is zero, so the sheaf is itself zero. 
\end{proof}

\subsection{The case of the constant sheaf}

It remains to prove that the upper bound on homological dimension is in fact an equality. 
To that end, we take the equivariant Godemont resolution of the constant sheaf at $\bQ$, $\const\bQ$.
\begin{align*}
\xymatrix{
0\ar[r]
&
\const\bQ
\ar[r]^-{\delta_0}
&
I^0 \ar[r]^-{\delta_1}
&
I^1 \ar[r]^-{\delta_2}
&
\ldots \ar[r] &
I^{n-1}\\
}
\end{align*}
Adding a small assumption on $X$ we can prove this resolution has length exactly $n-1$.

\begin{proposition}\label{prop:nonzerosections}
Let $G$ be a profinite group, let $X$ be a profinite scattered $G$-space of 
Cantor--Bendixson rank $n$ and assume that each $x \in X$ 
has a neighbourhood basis $\mcB_x$ of $\stab_G(x)$-invariant sets.
In the equivariant Godemont resolution of the constant sheaf, the cokernel of 
$\delta_i(U)$  has a non-zero $\stab_G(x)$-equivariant
section for each $U \in \mcB_x$ whenever $i$ is smaller than the height of $x$.
\end{proposition}
\begin{proof}
We start with the case of $\delta_0$ with $x$ of height at least 1.
Take a $\stab_G(x)$-invariant open neighbourhood of $x$. By 
Lemma \ref{lem:equivariantcantor} $U$ contains infinitely many points of other orbits
which are of lower height than $x$. 

Choose a non-zero element of $\const \bQ(U)_x=\bQ$ represented by  
a section $s \in \const \bQ(U)$.
We define a section $t$ of $I^0=\prod_{A} i^*_A E_{|A}$ by the sequence
\[
A \mapsto \begin{cases} 
0 & \textrm{$\text{ht}(A)$ has the same parity as $\text{ht}(x)$} \\
s_{|A} & \textrm{otherwise.}
\end{cases}
\]
Since $s$ is non-zero at infinitely many points near $x$, $t_x$ is non-zero. 
If $t=\delta_0(s')$ for some $s' \in \const \bQ(U)$ then $s'_{|Gx}=0$. 
This implies that there is an open neighbourhood of $x$ where 
$s'$ restricts to zero and hence $s'_y=0$ for all $y$ in that open neighbourhood, 
which implies that $t_x=0$, a contradiction.

The rest follows inductively, as with the non-equivariant case of \cite[Lemma 4.3]{sugrue19}, 
with two changes.
The first is that when we need to construct a new non-zero section we 
use the alternating process described previously.
The second is that since our sets $U$ are $\stab_G(x)$-invariant and we begin with a 
$\stab_G(x)$-equivariant section, all sections constructed in the proof are 
$\stab_G(x)$-equivariant. 

Note that the assumption on $i$ and the height of $x$ is required to ensure that the 
new section we construct is non-zero at infinitely many orbits. 
\end{proof}

We can now give the equivariant analogue of \cite[Theorem 4.4]{sugrue19}.

\begin{theorem}
Let $X$ be profinite $G$-space such that each $x \in X$ 
has a neighbourhood basis $\mcB_x$ of $\stab_G(x)$-invariant 
sets, for $G$ a profinite group.

If $X$ is a scattered $G$-space of Cantor--Bendixson rank $n$ and 
$x$ has height $n-1$ then 
\[
\ext^{n-1} \left(G\underset{\stab_G(x)}{\times} \bQ, \const\bQ\right) \neq 0.
\]
Hence, the homological dimension of the category of rational $G$-sheaves over $X$ is $n-1$. 

If $X$ has infinite Cantor--Bendixson rank then 
the homological dimension of the category of rational $G$-sheaves over $X$ is infinite.
\end{theorem}
\begin{proof}
We begin with a general calculation of maps out of an equivariant skyscraper sheaf
into $I^0(E)$ for $E$ some rational $G$-sheaf over $X$.
\begin{align*}
\hom \left(G\underset{\stab_G(x)}{\times} \bQ, I^0(E) \right) 
\cong &
\prod_{A} \hom \left(G\underset{\stab_G(x)}{\times} \bQ,  i^*_A E_{|A} \right) \\
\cong &
\hom \left(G\underset{\stab_G(x)}{\times} \bQ,  E_{|Gx} \right) \\
\cong &
\hom \left(G\underset{\stab_G(x)}{\times} \bQ,  G\underset{\stab_G(x)}{\times} E_x \right) \\
\cong &
E_x^{\stab_G{x}}\\
\end{align*}
The final term has fixed points as $\bQ$ has the trivial $\stab_G{x}$-action. 

Assume that $X$ is a scattered $G$-space of Cantor--Bendixson rank $n$ and 
$x$ has height $n-1$. 
Applying our calculation to our resolution of $\const \bQ$ we see that our $\ext$ groups are the homology
of the chain complex
\[
\bQ 
\overset{\alpha_0}{\lra} (\coker \delta_0)_x^{\stab_G(x)}
\overset{\alpha_1}{\lra} (\coker \delta_1)_x^{\stab_G(x)}
\overset{\alpha_2}{\lra}
\cdots
\overset{\alpha_{n-1}}{\lra}
(\coker \delta_{n-2})_x^{\stab_G(x)}. 
\]
By Proposition \ref{prop:nonzerosections} we have a non-zero $\stab_G(x)$-equivariant section
which implies that 
\[
(\delta_i)_x^{\stab_G(x)} \neq 0
\]
whenever $i$ is smaller than $n-1$.
By a similar argument to Proposition \ref{prop:nonzerosections}
we see that $\alpha_{n-1}$ is not surjective, hence the $n$th $\ext$ group is non-zero.
This calculation and Corollary \ref{cor:idupperbound} show that
the homological dimension is $n-1$. 

In the case of infinite Cantor--Bendixson rank we see that for each $n$ there is a 
point of height $n$, hence our earlier work shows that the homological dimension is infinite.
\end{proof}

The space of closed subgroups $\sub G$ of a profinite group $G$ 
always satisfies the condition on the invariant neighbourhood basis, see 
\cite[Section 2]{BSmackey}. 

\begin{corollary}\label{cor:injectivedimmain}
Let $G$ be a profinite group whose space of subgroups $\sub G$ is scattered of Cantor--Bendixson rank $n$.
The homological dimension of the category of rational Weyl-$G$-sheaves is $n-1$. 

If $\sub G$ has infinite Cantor--Bendixson rank, then the homological dimension 
of the category of rational Weyl-$G$-sheaves is infinite.
\end{corollary}
\begin{proof}
All the sheaves used in the proof of the theorem were stalk-wise fixed and hence were Weyl-$G$-sheaves.
\end{proof}

There is a remaining case of spaces which have finite Cantor--Bendixson rank and non-empty perfect hull.
We make an equivariant version of \cite[Conjecture 4.6]{sugrue19}.

\begin{conjecture}\label{conj:hull}
Let $G$ be a profinite group. If the $G$-space $X$ has finite Cantor--Bendixson rank and non-empty perfect hull, then the homological dimension 
of rational $G$-sheaves over $X$ is infinite.
\end{conjecture}

Using the equivalence between Weyl-$G$-sheaves and rational $G$-Mackey functors, 
\cite[Theorem A]{BSmackey}, we obtain the following calculation of homological dimensions
for categories of rational $G$-Mackey functors. 

\begin{corollary}
Let $G$ be a profinite group whose space of subgroups $\sub G$ is scattered of Cantor--Bendixson rank $n$.
The homological dimension of the category of rational $G$-Mackey functors is $n-1$. 

If $\sub G$ has infinite Cantor--Bendixson rank, then the homological dimension 
of the category of rational $G$-Mackey functors is infinite.
\end{corollary}

\begin{example}
Using our calculations from Example \ref{ex:rankexamples} we can give the homological dimension
of some algebraic models.

For $G$ a finite group, the homological dimension of rational Weyl-$G$-sheaves is zero. 
In terms of Mackey functors, this says that every rational $G$-Mackey functor
is injective, as was proven independently 
by two sources, Greenlees and May \cite[Appendix A]{gremay95} and Th\'evenaz
and Webb \cite[Theorems 8.3 and 9.1]{tw95}. See also \cite[Theorem 4.28]{BKalgmodels}
by the first author and K\k{e}dziorek.

For $G=\adic{p}$, the $p$-adic integers, the homological dimension of rational Weyl-$G$-sheaves is one,
which agrees with \cite[Lemma 6.2]{barnespadic}.
\end{example}

\bibliographystyle{alpha}
\bibliography{ourbib}

\end{document}